\documentclass{amsart}
\usepackage{amsmath,amssymb,amsbsy,amsthm,amsfonts,mathtools,hyperref,color}

\begin{document}

\title[Global Lipschitz estimates for nonstandard growth]{A boundary regularity result for minimizers of variational integrals with nonstandard growth}
\author[ M.~Bul\'{\i}\v{c}ek]{Miroslav Bul\'{\i}\v{c}ek}
\address{Charles University, Faculty of Mathematics and Physics, Mathematical Institute, Sokolovsk\'{a} 83, 186~75, Prague, Czech Republic}
\email{mbul8060@karlin.mff.cuni.cz}

\author[E.~Maringov\'{a}]{Erika Maringov\'{a}}
\address{Charles University, Faculty of Mathematics and Physics, Mathematical Institute, Sokolovsk\'{a} 83, 186~75, Prague, Czech Republic}
\email{maringova@karlin.mff.cuni.cz}

\author[B.~Stroffolini]{Bianca Stroffolini}
\address{Dipartimento di Matematica e Applicazioni, Universit\`{a} di Napoli ''Federico II", via Cintia, 80126 Napoli, Italy}
\email{bstroffo@unina.it}

\author[A.~Verde]{Anna Verde}
\address{Dipartimento di Matematica e Applicazioni, Universit\`{a} di Napoli "Federico II", via Cintia, 80126 Napoli, Italy}
\email{anverde@unina.it}

\thanks{M.~Bul\'{\i}\v{c}ek's  was  supported
by the Czech Science Foundation (grant no. 16-03230S) and he is a member of the Ne\v{c}as center for Mathematical Modeling. E.~Maringov\'{a} thanks to the project  SVV-2017-260455.  The other authors are members of the Gruppo Nazionale per l'Analisi
Matematica, la Probabilit\`a e le loro Applicazioni (GNAMPA) of the
Istituto Nazionale di Alta Matematica (INdAM)}

\dedicatory{Dedicated to Professor Carlo Sbordone on the occasion of his 70th birthday.}


\begin{abstract}
We prove global Lipschitz regularity for a wide class of convex variational integrals  among all functions in $W^{1,1}$ with prescribed (sufficiently regular) boundary values, which are not assumed to satisfy any geometrical constraint (as for example bounded slope condition). Furthermore, we do not assume any restrictive assumption on the geometry of the domain and the result is valid for all sufficiently smooth domains.  The result is achieved with a suitable approximation of the functional together with a new construction of  appropriate barrier functions.
\end{abstract}

\maketitle

\newcommand{\ddd}{\,{\rm d}}
%\numberwithin{equation}{section} \marginparwidth=2cm

\def\note#1{\marginpar{\small #1}}
\def\tens#1{\pmb{\mathsf{#1}}}
\def\vec#1{\boldsymbol{#1}}
\def\norm#1{\left|\!\left| #1 \right|\!\right|}
\def\fnorm#1{|\!| #1 |\!|}
\def\abs#1{\left| #1 \right|}
\def\ti{\text{I}}
\def\tii{\text{I\!I}}
\def\tiii{\text{I\!I\!I}}

\newcommand{\loc}{{\rm loc}}
\def\diver{\mathop{\mathrm{div}}\nolimits}
\def\grad{\mathop{\mathrm{grad}}\nolimits}
\def\Div{\mathop{\mathrm{Div}}\nolimits}
\def\Grad{\mathop{\mathrm{Grad}}\nolimits}
\def\tr{\mathop{\mathrm{tr}}\nolimits}
\def\cof{\mathop{\mathrm{cof}}\nolimits}
\def\det{\mathop{\mathrm{det}}\nolimits}
\def\lin{\mathop{\mathrm{span}}\nolimits}
\def\pr{\noindent \textbf{Proof: }}

\def\pp#1#2{\frac{\partial #1}{\partial #2}}
\def\dd#1#2{\frac{\d #1}{\d #2}}
\def\bA{\tens{A}}
\def\T{\mathcal{T}}
\def\R{\mathcal{R}}
\def\bx{\vec{x}}
\def\be{\vec{e}}
\def\bef{\vec{f}}
\def\bec{\vec{c}}
\def\bs{\vec{s}}
\def\ba{\vec{a}}
\def\bn{\vec{n}}
\def\bphi{\vec{\varphi}}
\def\btau{\vec{\tau}}
\def\bc{\vec{c}}
\def\bg{\vec{g}}
\def\mO{\mathcal{O}}
\def\pmO{\partial\mathcal{O}}
\def\bE{\tens{\varepsilon}}
\def\bsig{\tens{\sigma}}
\def\bW{\tens{W}}
\def\bT{\tens{T}}
\def\bxi{\tens{\xi}}
\def\bD{\tens{D}}
\def\bF{\tens{F}}
\def\bB{\tens{B}}
\def\bV{\tens{V}}
\def\bS{\tens{S}}
\def\bI{\tens{I}}
\def\bi{\vec{i}}
\def\bv{\vec{v}}
\def\bfi{\vec{\varphi}}
\def\bk{\vec{k}}
\def\b0{\vec{0}}
\def\bom{\vec{\omega}}
\def\bw{\vec{w}}
\def\p{\pi}
\def\bu{\vec{u}}
\def\bz{\vec{z}}
\def\bep{\vec{e}_{\textrm{p}}}
\def\dbep{\dot{\vec{e}}_{\textrm{p}}}
\def\bee{\vec{e}_{\textrm{el}}}
\def\dbee{\dot{\vec{e}}_{\textrm{el}}}
\def\ID{\mathcal{I}_{\bD}}
\def\IP{\mathcal{I}_{p}}
\def\Pn{(\mathcal{P})}
\def\Pe{(\mathcal{P}^{\eta})}
\def\Pee{(\mathcal{P}^{\varepsilon, \eta})}
\def\dbx{\,{\rm d} \bx}
\def\dx{\,{\rm d}x}
\def\dr{\,{\rm d}r}
\def\ds{\,{\rm d}s}
\def\dt{\,{\rm d}t}
\def\omer{\omega^{q,r_0}}
\def\ber{b^q}
\def\ver{v^{q,r_0}_{\bk, c}}
\def\hf{\hat{f}}

\def\F{\tilde F}
\def\a{\tilde a}
%------------------------------------------------

\newtheorem{Theorem}{Theorem}[section]

\newtheorem{Example}{Example}[section]
\newtheorem{Lemma}{Lemma}[section]
\newtheorem{Rem}{Remark}[section]
\newtheorem{Def}{Definition}[section]
\newtheorem{Col}{Corollary}[section]

\numberwithin{equation}{section}

%\maketitle

\section{Introduction}
In this paper we are concerned with the existence of (unique) scalar-valued Lipschitz solutions to the Dirichlet problem
\begin{equation}\label{P1}
\begin{aligned}
-\diver \left(a(|\nabla u|)\nabla u \right)&=0 &&\textrm{ in } \Omega,\\
u&=u_0 &&\textrm{ on } \partial \Omega,
\end{aligned}
\end{equation}
where $\Omega \subset \mathbb{R}^d$ is a bounded, regular domain and with regular prescribed boundary values~$u_0$. In this setting, the existence of a weak solution to the Dirichlet problem~\eqref{P1} is equivalent to the existence of a minimizer of a related (convex) variational integral in the Dirichlet class $u_0 + W^{1,1}_0(\Omega)$, and we may equivalently look for a function $u \in u_0 + W^{1,1}_0(\Omega)$ such that for all smooth, compactly supported test function $\varphi \in \mathcal{D}(\Omega)$ we have
\begin{equation}\label{min}
\int_{\Omega}F(|\nabla u|)\dx \le \int_{\Omega}F(|\nabla u_0 +\nabla \varphi|)\dx,
\end{equation}
where $F$ and $a$ are linked via the identity
\begin{equation}
F'(s)=a(s)s \qquad \text{for all } s \in \mathbb{R}^+.
\label{min2}
\end{equation}

For the later use, it is convenient to summarize the formulas following from \eqref{min2} at this place,
\begin{equation}\label{af}
\begin{aligned}
a(s) &= \frac{F'(s)}{s}, \hspace{1cm} a'(s) = \frac{F''(s)}{s} - \frac{F'(s)}{s^2}, \\
s \frac{a'(s)}{a(s)} &= s \frac{F''(s)}{F'(s)} -1,
\end{aligned}
\end{equation}
which are valid for any $s \in \mathbb{R}^+$.

We study the existence of Lipschitz minimizers of the problem \eqref{min} in a wide class of convex variational integrals ranging from nearly linear growth right up to exponential one including also these borderline cases. We can also treat  functionals  with  the so-called $(p, q)$-structure, see \cite{Ma}.
A classical example of oscillating function between $p$ and $q$ growth is  the following function $F:[0,\infty)\to [0,\infty)$
\begin{equation}\label{example}
F(t):=
\left\{\begin{array}{ll}
\displaystyle t^p &\text{if $ 0\leq t\leq t_0$,}
\\
\displaystyle t^{(\frac{p+q}{2}+\frac{p-q}{2}\sin\log\log\log t)}&\text{if $t>t_0$,}
\end{array}\right.
\end{equation}
where $t_0>0$ is chosen so that $\sin \log \log \log t _0
= 1$ (this function was first given as an example in \cite{DMP}).

For scalar functions the Lipschitz continuity of minimizers has been investigated using the bounded slope condition or the barrier functions.
When $F$ is convex, the validity of the so-called bounded slope condition (BSC) due to Hartmann, Nirenberg and Stampacchia (which is a geometrical assumption on the boundary data $u_0$) ensures the existence of a minimizer among Lipschitz functions, see \cite{S}. In addition, if $F$  is also strictly convex, every continuous $W^{1,1}$- minimizer is Lipschitz continuous on $\Omega$, see \cite{Cellina,Cellina2}. Local Lipschitz regularity of solutions was proved by F.~H.~Clarke in \cite{Clarke} for strictly convex, $p$-coercive ($p>1$) functions F under a weaker condition on $u_0$, the so-called lower bounded slope condition, a condition corresponding to the left-sided version of the BSC, see also \cite{MT}.

Let us also mention the Perron method \cite{Perron}, originally developed
for the Laplace equation, and generalized for analyzing Dirichlet boundary value problems
to various elliptic partial differential equations. The idea is to
construct an upper solution of the Dirichlet problem as an infimum of a certain upper class of
supersolutions. A lower solution is constructed similarly using a lower class of subsolutions,
and when the upper and lower solutions coincide we obtain a  solution. In the
Perron method, the boundary regularity is essentially a separate problem from the existence
of a solution. The use  of barrier functions as a tool for studying boundary regularity seems to
go back to the Lebesgue  paper \cite{L1}. In  \cite{L2}  Lebesgue characterised regular
boundary points in terms of barriers for the linear Laplace equation.
The extension of Perron's method and the method of barriers to the nonlinear $p$-Laplacian
was initiated by Granlund, Lindqvist and Martio in \cite{GLM} and developed in a series of papers
(see, for example, the accounts given in Heinonen et al. \cite{HKM}).

%The Perron method (Harmonic replacement) for solving the Dirichlet problem associated to the Laplace equation is classical. It is based on the maximum principle and the submean value property for subharmonic functions.
%It has been extended also to nonlinear operators as the $p$-Laplacian, see the book \cite{HKM}.

In this paper, the attainment of the boundary data is performed by constructing a barrier at regular points.
It seems  that the simplest condition for regular points is the {\it exterior ball condition}.
\begin{Def}\label{omega}
A domain $\Omega$ satisfies the uniform exterior ball condition if there exists a number $r_0>0$
such that for every point $x_0\in \partial\Omega$
 there is a ball $B_{r_0}(x_0)$ such that $B_{r_0}(x_0)\cap \Omega=\{x_o\}$.
\end{Def}
\begin{Rem}\label{RRMB}
Convexity or $C^{1,1}$-regularity of the domain are sufficient for the uniform exterior ball condition,
see e.g. \cite{GT}, thus, Theorem \ref{T1} holds in particular for all convex domains
of class $C^1$ and for arbitrary domains of class $C^{1,1}$.
\end{Rem}

In the linear case the method of the barrier function for domains satisfying exterior ball condition has been presented in  \cite{BBM}. More precisely, in \cite{BBM} the authors study the minimization of convex, variational integrals of linear growth. Due to insufficient compactness properties of these Dirichlet classes,
the existence of solutions does not follow in a standard way by the direct method in
the calculus of variations. Assuming radial structure, they establish a necessary
and sufficient condition on the integrand such that the Dirichlet problem is in general
solvable, in the sense that a Lipschitz solution exists for any regular domain and all
prescribed regular boundary values, via the construction of appropriate barrier functions
in the spirit of Serrin's paper \cite{S}.

In this paper, we significantly generalize the method used in \cite{BBM} and we are able to treat also the case of variational integrals ranging from nearly linear growth right up to exponential one. Our main result is the following.
\begin{Theorem}\label{T1}
Let $F\in \mathcal{C}^2(\mathbb{R}^+)$ be a strictly convex function with $\lim_{s \to 0} F'(s) = 0$ which satisfies, for some constants $C_1,C_2 >0$,
\begin{align}
C_1s -C_2 &\le F(s)~~~\text{ for all } s\in \mathbb{R}^+,\label{A1}\\
\liminf_{s \to \infty} \frac{s^{2-\delta}F''(s)}{F'(s)}&\ge 2. \label{A2}
\end{align}
Then for arbitrary domain $\Omega$ of class $\mathcal{C}^{1}$ satisfying the uniform exterior ball condition and arbitrary prescribed boundary value $u_0 \in \mathcal{C}^{1,1}(\overline{\Omega})$ there exists a unique function $u\in \mathcal{C}^{0,1}(\overline{\Omega})$ solving~\eqref{P1}.
\end{Theorem}

\begin{Rem}
In fact \eqref{A2} can be relaxed and replaced by
\begin{align}
\liminf_{s \to \infty} \frac{s^{2}F''(s)}{(\ln s)^{1+\delta}F'(s)}&\ge 2. \label{A222}
\end{align}
\end{Rem}

Let us emphasize the key novelty of the result. First, we do not require any geometrical constraint on boundary data and/or on the domain and the result is valid for all $\mathcal{C}^{1,1}$ domains and arbitrary $u_0\in \mathcal{C}^{1,1}(\overline{\Omega})$. Furthermore, we do not assume any specific growth condition on $F$ as all we need is the sufficient convexity assumption \eqref{A2}. It is worth noticing that the assumption \eqref{A2} is not only sufficient for getting global Lipschitz  solutions but also necessary for $F$ having linear growth, as  it is  shown in \cite{BBM}. Last, we are also able to cover the case of logarithmic, exponential or even oscillating growth condition for $F$, see e.g. the example given  in \eqref{example}, which also satisfies \eqref{A2}. Furthermore, we are even able to go beyond the logarithmic or exponential growth. Indeed, we define
$$
F(s)=s\eta(s) \qquad \textrm{ with } 1\le \eta(s) \overset{s\to \infty}{\to} \infty,
$$
where $\eta$ is smooth non-decreasing and fulfils for all $s>0$
$$
2\eta'(s) + s\eta''(s) >0.
$$
Then $F$ is strictly convex and satisfies \eqref{A1}. The condition \eqref{A2} is equivalent to
$$
\liminf_{s \to \infty} \frac{s^{2-\delta}(2\eta'(s) +s\eta''(s))}{\eta(s)+s\eta'(s)}\ge 2.
$$
Finally, choosing  $\eta(s)$ for example such that $\eta(s)\sim e^{e^{s}}$ as $s\to \infty$ or such that for some $\alpha>0$ we have $\eta(s)\sim \ln^{\alpha}s$ as $s\to \infty$, then \eqref{A2} remains valid. Hence, we see that even faster growth than exponential or slower growth than logarithmic are covered by our result.

The proof will be given by an ``approximation" scheme.  First, we approximate $F$ with functionals $F_{\lambda}$ that are quadratic for large values of $\lambda$. In particular, they are strictly convex and so they admit unique minimizers $u_{\lambda}$. Then, we construct lower and upper barriers to $u_{\lambda}$  using an appropriate auxiliary problem. The link with our original problem (true barrier function, see section $2.5$) is achieved with the selection of the parameter~$\lambda$ large enough in order to guarantee the upper bound. Using the barrier function, we are able to get uniform Lipschitz estimates and get the result in the limit.

\section{Proof of  Theorem~\ref{T1}}
\label{main_proof}

This section is devoted to the proof of the result of this paper. The proof will be given by an ``approximation" scheme. This means that for some $\lambda>0$ we shall approximate the original $F$ by $F_{\lambda}$, which will still fulfill \eqref{A1}--\eqref{A2}, will satisfy $F_{\lambda}(s)=F(s)$ for all $s\le \lambda$ but will be quadratic for all $s\ge \lambda$. For such chosen  $\lambda$, we find a minimizer $u_{\lambda}$ to
\begin{equation}\label{minl}
\int_{\Omega}F_{\lambda}(|\nabla u_{\lambda}|)\dx \le \int_{\Omega}F_{\lambda}(|\nabla u_0 +\nabla \varphi|)\dx
\end{equation}
and introduce also the corresponding $a_{\lambda}$ via the identity
\begin{equation}
F'_{\lambda}(s)=a_{\lambda}(s)s \qquad \text{for all } s \in \mathbb{R}^+.
\label{min2l}
\end{equation}
Then the minimizer $u_{\lambda}$ will also solve
\begin{equation}\label{P1l}
\begin{aligned}
-\diver \left(a_{\lambda}(|\nabla u_{\lambda}|)\nabla u_{\lambda} \right)&=0 &&\textrm{ in } \Omega,\\
u_\lambda&=u_0 &&\textrm{ on } \partial \Omega.
\end{aligned}
\end{equation}
Finally,  our goal will be to specify  $\lambda>0$ for which there holds
\begin{equation}
\|\nabla u_{\lambda}\|_{\infty}\le \lambda. \label{dream}
\end{equation}
Then we immediately have that (since $a_{\lambda}(s)=a(s)$ for all $s\le \lambda$)
$$
\diver \left(a(|\nabla u_{\lambda}|)\nabla u_{\lambda} \right)=\diver \left(a_{\lambda}(|\nabla u_{\lambda}|)\nabla u_{\lambda} \right)\overset{\eqref{P1l}}=0
$$
and consequently, $u_{\lambda}$ is a solution to \eqref{P1} and therefore also a minimizer to \eqref{min}. Then due to the uniqueness of the minimizer, we get the claim of Theorem~\ref{T1}.

\subsection{Approximation $F_{\lambda}$}\label{SS23}
First, we fix some $\lambda_0$ such that for all $s\ge \lambda_0$ the second derivative of $F$ exists and is positive. Note that the existence of such $\lambda_0$ is a consequence of assumption \eqref{A2}. Indeed, we can set $\lambda_0$ in such a way that
\begin{equation}
\frac{s^{2-\delta}F''(s)}{F'(s)} \ge 1 \label{A2b}
\end{equation}
for all $s\ge \lambda_0$. Then for arbitrary $\lambda\ge \lambda_0$ we define the approximative $F_{\lambda}$ as follows
\begin{equation}
F_{\lambda}(s):= \left\{\begin{aligned}
&F(s) &&\textrm{for }s\le \lambda,\\
&F(\lambda) +F'(\lambda)(s-\lambda) + \frac12 F''(\lambda)(s-\lambda)^2 &&\textrm{for } s>\lambda.
\end{aligned}
\right. \label{Dfla}
\end{equation}
Direct computation  leads to
\begin{align}
F'_{\lambda}(s)&=\left\{ \begin{aligned}
&F'(s) &&\textrm{for }s\le \lambda,\\
&F'(\lambda) + F''(\lambda)(s-\lambda) &&\textrm{for } s>\lambda.
\end{aligned}\right.\label{1Fd}\\
F''_{\lambda}(s)&=\left\{\begin{aligned}
&F''(s) &&\textrm{for }s\le \lambda,\\
&F''(\lambda) &&\textrm{for } s>\lambda.
\end{aligned} \right. \label{2Fd}
\end{align}
With such a definition, it is clear that $F_{\lambda}$ satisfies \eqref{A1}--\eqref{A2} with constants $C_1$ and $C_2$. In addition, we see that
\begin{equation}
C_3(\lambda)s^2 - C_4(\lambda) \le F_{\lambda}(s) \le C_4(\lambda)(s^2+1)\label{c3c4}
\end{equation}
and that $F_{\lambda}$ is strictly convex. Therefore by using the standard methods of calculus of variations there exists unique $u_{\lambda} \in W^{1,2}(\Omega)$ solving \eqref{minl}. Our goal is to show \eqref{dream} provided that $\lambda$ is chosen properly.

The function $F_{\lambda}$ is still strictly convex but not necessarily uniformly. Therefore, we introduce next level of approximation, namely
\begin{equation}
F_{\lambda,\mu}(s)=\frac{\mu}{2} s^2 + F_{\lambda}(s). \label{mu}
\end{equation}
This function still satisfies \eqref{c3c4} with a possibly different constants $C_3$ and $C_4$ but is uniformly convex. Therefore we can find $u_{\lambda,\mu}\in W^{1,2}(\Omega)$ that solves
\begin{equation}\label{minm}
\int_{\Omega}F_{\lambda,\mu}(|\nabla u_{\lambda,\mu}|)\dx \le \int_{\Omega}F_{\lambda,\mu}(|\nabla u_0 +\nabla \varphi|)\dx.
\end{equation}
Moreover, due to the standard maximum principle we also have
\begin{equation}
\|u_{\lambda ,\mu}\|_{\infty} \le \|u_0\|_{\infty}.\label{AE3}
\end{equation}
In addition, we see that
$$
u_{\lambda,\mu} \to u_{\lambda} \textrm{ in } W^{1,2}(\Omega)
$$
as $\mu \to 0_+$. Therefore, if we show that
\begin{equation}
\|\nabla u_{\lambda,\mu}\|_{\infty}\le \lambda, \label{muin}
\end{equation}
then \eqref{dream} follows. Thus, it remains to find some $\lambda \ge \lambda_0$ and some $\mu_0$ such that for all $\mu \in (0,\mu_0)$ the estimate \eqref{muin} holds.

\subsection{Second derivatives and maximum principle} Starting from this subsection, we omit writing subscripts in $u$ to shorten the notation, i.e., we denote $u:=u_{\lambda,\mu}$, where $u_{\lambda,\mu}$ is the unique minimizer to \eqref{minm}.

Due to the definition of $F_{\lambda,\mu}$, we see that it is uniformly convex. Therefore we can use the classical result and due to the regularity of the domain $\Omega$ and the boundary data $u_0$, we know that $u \in W^{1,\infty}(\Omega) \cap W^{2,2}(\Omega)$. In addition, defining $a_{\lambda}$ by
$$
F'_{\lambda}(s)=s a_{\lambda}(s)
$$
we see that $u$ solves
\begin{equation}\label{P1lm}
\begin{aligned}
-\mu \Delta u - \diver \left(a_{\lambda}(|\nabla u|)\nabla u \right)&=0 &&\textrm{ in } \Omega,\\
u&=u_0 &&\textrm{ on } \partial \Omega.
\end{aligned}
\end{equation}
Due to the $W^{2,2}$ regularity of the solution, we may now apply $D_k:=\partial_{x_k}$ onto the equation and multiply the result by $D_k u$ and sum over $k=1,\ldots, d$ to obtain (we use the Einstein summation convention)
\begin{align*}
0&=D_k\left(-\mu \Delta u - \diver \left(a_{\lambda}(|\nabla u|)\nabla u \right)\right)D_ku \\
%&= -D_i\left(\mu  D_{ki} u D_k u + D_k\left(a_{\lambda}(|\nabla u|)D_i u \right)D_k u\right) + \mu |\nabla^2 u|^2 +D_k\left(a_{\lambda}(|\nabla u|)\right)D_{ki} u \\
&= -\frac{\mu}{2}  \Delta |\nabla u|^2 -D_i\left(\left( a'_{\lambda}(|\nabla u|)D_k|\nabla u| D_i u +a_{\lambda}(|\nabla u|)D_{ki} u \right)D_k u\right)\\
 &\quad + \mu |\nabla^2 u|^2 +\left( a'_{\lambda}(|\nabla u|)D_k|\nabla u| D_i u +a_{\lambda}(|\nabla u|)D_{ki} u \right) D_{ki} u\\
&= -\frac{\mu}{2}  \Delta |\nabla u|^2 -\frac12 D_i\left( a'_{\lambda}(|\nabla u|)|\nabla u| D_k|\nabla u|^2 \frac{D_i u D_k u}{|\nabla u|^2}+a_{\lambda}(|\nabla u|)D_{i} |\nabla u|^2 \right)\\
&\quad + \mu |\nabla^2 u|^2 + a'_{\lambda}(|\nabla u|)|\nabla u| |\nabla |\nabla u||^2  +a_{\lambda}(|\nabla u|)|\nabla^2 u|^2.
\end{align*}
Next, using the fact that (which follows from \eqref{af}, where we replace $F$ by $F_{\lambda}$)
\begin{equation}\label{giveme}
a'_{\lambda}(s)s = (a_{\lambda}(s)s)' - a_{\lambda} = F''_{\lambda}(s) - \frac{F'_{\lambda}(s)}{s},
\end{equation}
we can rewrite the above identity as
\begin{align*}
0&= -\frac{\mu}{2}  \Delta |\nabla u|^2 + \mu |\nabla^2 u|^2 -\frac12 D_i\left(  F''_{\lambda}(|\nabla u|) D_k|\nabla u|^2 \frac{D_i u D_k u}{|\nabla u|^2} \right)\\
 &\quad-\frac12 D_i\left(\frac{F'_{\lambda}(|\nabla u|)}{(|\nabla u|)}\left(D_{i} |\nabla u|^2-  D_k|\nabla u|^2 \frac{D_i u D_k u}{|\nabla u|^2}\right) \right)\\
&\quad  +  F''_{\lambda}(|\nabla u|) |\nabla |\nabla u||^2  +\frac{F'_{\lambda}(|\nabla u|)}{(|\nabla u|)}(|\nabla^2 u|^2-|\nabla |\nabla u||^2)\\
&\ge -\frac{\mu}{2}  \Delta |\nabla u|^2 -\frac12 D_i\left(  F''_{\lambda}(|\nabla u|) D_k|\nabla u|^2 \frac{D_i u D_k u}{|\nabla u|^2} \right)\\
&\qquad-\frac12 D_i\left(\frac{F'_{\lambda}(|\nabla u|)}{(|\nabla u|)}\left(D_{i} |\nabla u|^2-  D_k|\nabla u|^2 \frac{D_i u D_k u}{|\nabla u|^2}\right) \right)\\
&=: -\frac{\mu}{2}  \Delta |\nabla u|^2 -\frac12 D_i\left(  a_{ik} D_k|\nabla u|^2  \right),
\end{align*}
where for the inequality we used the convexity of $F_{\lambda}$. Note also that due to the convexity of $F_{\lambda}$, the matrix $a_{ij}$ is positively semidefinite, i.e., for arbitrary $\xi \in \mathbb{R}^d$, there holds
\begin{equation}\label{definit}
a_{ij} \xi_i \xi_j \ge 0.
\end{equation}
Finally, if we multiply the resulting inequality by $\max\{0, |\nabla u|^2-\|\nabla u\|_{L^{\infty}(\partial \Omega)}^2\}$, and integrate by parts (note here that due to the regularity of $u$, such a procedure is rigorous) and use the fact that the boundary integral vanishes, we deduce that
$$
\int_{\Omega}|\nabla \max\{0, |\nabla u|^2-\|\nabla u\|_{L^{\infty}(\partial \Omega)}^2\}|^2 \dx =0,
$$
which consequently implies that
\begin{equation*}
\|\nabla u\|_{L^{\infty}}\le \|\nabla u\|_{L^{\infty}(\partial \Omega)}.
\end{equation*}
Finally, since $u=u_0$ on $\partial \Omega$, we can simplify the above estimate  to
\begin{equation}
\|\nabla u\|_{L^{\infty}}\le \|\nabla u_0\|_{L^{\infty}(\partial \Omega)}+ \left\|\partial_{\bn} u\right\|_{L^{\infty}(\partial \Omega)}\le C+ \left\|\partial_{\bn} u\right\|_{L^{\infty}(\partial \Omega)},\label{supersolu}
\end{equation}
where $\partial_{\bn}$ denotes the normal derivative of $u$ on $\partial \Omega$ and $C$ is a constant, which is independent of $\lambda$ and $\mu$. Hence, to prove \eqref{muin}, we need to show that
\begin{equation}
\|\partial_{\bn}u \|_{L^\infty(\partial \Omega)}\le C\label{basis2}
\end{equation}
with $C$ being independent of $\lambda$ and $\mu$. Indeed, if \eqref{basis2} holds true, then it also follows from \eqref{supersolu} that \eqref{muin} holds provided that $\lambda \ge 2C$. The rest of the paper is devoted to the proof of~\eqref{basis2}, which will be shown via the barrier function technique.

\subsection{Estimates of normal derivatives via barrier functions} \label{barrier}
Our goal is to show that for almost all $\bx \in \partial \Omega$ there holds
\begin{equation}
|\partial_{\bn} u(\bx)|\le C \label{numb}
\end{equation}
with a constat $C$ independent of $\lambda$ and $\mu$. Notice that since $u\in W^{2,2}(\Omega)$, we know that it makes sense to consider $\nabla u$ on $\partial \Omega$. Assume for a moment that for given $\bx_0 \in \partial \Omega$ we can find $u^b$ and $u_b$ such that
$$
u_b(\bx_0)=u^b(\bx_0)=u_0(\bx_0)
$$
and fulfilling for all $\bx\in B_{r}(\bx_0)\cap \Omega$ with some $r>0$
\begin{equation}\label{truee}
u_b(\bx)\le u(\bx)\le u^b(\bx).
\end{equation}
Then we have
\begin{align*}
\partial_{\bn} u(\bx_0)&\le \limsup_{\{\bx \in \Omega; \, \bx \to \bx_0\}}\frac{u(\bx)-u(\bx_0)}{|\bx - \bx_0|}\\
&\le  \limsup_{\{\bx \in \Omega; \, \bx \to \bx_0\}}\frac{u^b(\bx)-u(\bx_0)}{|\bx - \bx_0|}+ \limsup_{\{\bx \in \Omega; \, \bx \to \bx_0\}}\frac{u(\bx)-u^b(\bx)}{|\bx - \bx_0|}\\
&\le |\nabla u^b(\bx_0)|\\
\partial_{\bn} u(\bx_0)&\ge  \liminf_{\{\bx \in \Omega; \, \bx \to \bx_0\}}\frac{u(\bx)-u(\bx_0)}{|\bx - \bx_0|}\\
&\ge   \liminf_{\{\bx \in \Omega; \, \bx \to \bx_0\}}\frac{u_b(\bx)-u(\bx_0)}{|\bx - \bx_0|}+ \liminf_{\{\bx \in \Omega; \, \bx \to \bx_0\}}\frac{u(\bx)-u_b(\bx)}{|\bx - \bx_0|}\\
&\ge -|\nabla u_b(\bx_0)|.
\end{align*}
Consequently
\begin{equation}\label{wened}
|\partial_{\bn} u(x_0)|\le \max\{|\nabla u_b(\bx_0)|, |\nabla u^b(\bx_0)|\}.
\end{equation}
Thus, we can reduce everything just on finding proper \emph{barriers} $u_b$ and $u^b$ for which we can control derivatives independently of the choice of $\bx_0$. This will be however done by looking for sub- and super-solutions to an original problem. Hence if we succeed in finding $u_b$ and $u^b$ such that
\begin{equation}\label{ubsub}
\begin{aligned}
-\mu \Delta u^b - \diver \left(a_{\lambda}(|\nabla u^b|)\nabla u^b \right)&\ge 0 &&\textrm{ in } \Omega\cap B_r(\bx_0),\\
u^b&\ge u  &&\textrm{ on } \partial (\Omega\cap B_r(\bx_0)),
\end{aligned}
\end{equation}
\begin{equation}\label{ubsup}
\begin{aligned}
-\mu \Delta u_b - \diver \left(a_{\lambda}(|\nabla u_b|)\nabla u_b \right)&\le 0 &&\textrm{ in } \Omega\cap B_r(\bx_0),\\
u_b&\le u &&\textrm{ on } \partial (\Omega\cap B_r(\bx_0))
\end{aligned}
\end{equation}
and satisfying $u_b(\bx_0)=u^b(\bx_0)=u_0(\bx_0)$ then due to the convexity of $F_{\lambda}$ we can deduce \eqref{truee}. Thus, it just remains to construct solutions to \eqref{ubsub} and \eqref{ubsup} (and consequently to fix $r>0$) for which we are able to control gradients independently of $\bx_0$, $\lambda$ and $\mu$. Since the procedure of finding barriers $u_b$ and $u^b$ is in fact the same, we focus in what follows only on finding a function fulfilling~\eqref{ubsub}.

\subsection{Prototype barrier function}
The prototype barrier function will be found as a solution to a special problem. We shall define
\begin{equation}\label{ag}
\frac{\F'(s)}{s}= \a(s) := \frac{1}{1+s}>0 \qquad \text{for}~s>0.
\end{equation}
and we can find the corresponding $\F(s):=\int_0^s \F'(t) \dt$, which is convex. In addition, we have for all $s>0$,
\begin{equation}\label{propag}
\a'(s) := -\frac{1}{(1+s)^2}<0 \qquad \text{and} \qquad s \frac{\a'(s)}{a(s)} = -\frac{s}{1+s}<0.
\end{equation}
The minimization problem for $\F$ then serve as a kind of comparison problem to the minimization of $F_{\lambda,\mu}$. Using the definition of $\a$, we see that  $\F'$ is a strictly monotonically increasing mapping from $[0,\infty)$ to $[0,1)$ with continuous inverse. With the help of~$\F$ we now define our prototype barrier function.

For arbitrary $r_0>0$ and $q\in (0, r_0^{d-1})$, we set
\begin{equation}\label{DFb}
\ber(r)\coloneqq (\F')^{-1} \left(\frac{q}{r^{d-1}}\right)=\frac{q}{r^{d-1}-q}.
\end{equation}
It can be easily seen that $\ber \in \mathcal{C}^1[0,\infty)$ is a non-negative decreasing function. Finally, for all $x\in \mathbb{R}^d\setminus B_{r_0}(\b0)$, $r_0>0$, we define
\begin{equation}
\omer(\bx)\coloneqq \int_{r_0}^{|\bx|} \ber(r)\dr. \label{dfomer}
\end{equation}

By construction,~$\omer$ is a minimizer of the functional with integrand $\F$ and equivalently a solution to the associated Dirichlet problem on the set $\mathbb{R}^d \setminus \overline{B_{r_{0}}(\b0)}$, but moreover, it also turns out to be super-harmonic. To summarize, we have the following result.
\begin{Lemma}\label{L3}
For every $r_0>0$ and $q\in (0,r_0^{d-1})$ the function~$\omer$ defined in~\eqref{dfomer} satisfies
\begin{equation}\label{P1v}
\begin{aligned}
-\diver \left( \a(|\nabla \omer|)\nabla \omer \right) &=0 &&\textrm{ in } \mathbb{R}^d \setminus \overline{B_{r_{0}}(\b0)},\\
\omer &=0 &&\textrm{ on } \partial B_{r_{0}}(\b0).
\end{aligned}
\end{equation}
Furthermore, there holds
\begin{equation}\label{subhar}
-\Delta \omer (\bx) \ge 0 \qquad \textrm{for all } \bx \in \mathbb{R}^d \setminus \overline{B_{r_{0}}(\b0)}.
\end{equation}
\end{Lemma}

\begin{proof}
Using the definition of~$\omer$, we immediately see that~$\omer$ vanishes on $\partial B_{r_{0}}(\b0)$, and we further observe
\begin{equation}
\label{omerp}
\nabla \omer(\bx) = \ber(|\bx|)\frac{\bx}{|\bx|} \quad \text{and} \quad |\nabla \omer(\bx)| = \ber(|\bx|).
\end{equation}
Via the definition of $\ber$, we thus have
\begin{equation*}
\F'(|\nabla \omer(\bx)|) = \F'(\ber(|\bx|))= \frac{q}{|\bx|^{d-1}}.
\end{equation*}
Consequently, for all $|\bx|>r_0$ there holds
\begin{equation}
\begin{split}
\diver \left(\a(|\nabla \omer(\bx)|)\nabla \omer (\bx)\right)&=\diver\Big(\F'(|\nabla \omer(\bx)|)\frac{\nabla \omer(\bx)}{|\nabla \omer (\bx)|}\Big) \\
  &=q \diver \frac{\bx }{|\bx|^{d}}=0
\end{split}
\end{equation}
and the solution property~\eqref{P1v} follows. Finally, we check the super-harmonicity property of~$\omer$. In view of~\eqref{P1v} and~\eqref{omerp} we get
$$
\begin{aligned}
0&=\diver \left( \a(|\nabla \omer(\bx)|)\nabla \omer(\bx) \right)\\
&=\a(|\nabla \omer(\bx)|)\Delta \omer(\bx) + \nabla \a(|\nabla \omer(\bx)|)\cdot \nabla \omer(\bx)\\
&=\a(\ber(|\bx|))\Delta \omer(\bx) + \a'(\ber(|\bx|))\ber(|\bx|)(\ber)'(|\bx|).
\end{aligned}
$$
Therefore, since the functions~$\a$ and~$\ber$ are positive and~$\ber$ is monotonically decreasing, also the second claim~\eqref{subhar} follows.
\end{proof}

Thus,~$\omer$ is a good prototype super-solution to the approximative problem on a certain set. However, due to the possibly non-constant prescribed boundary values~$u_0$, it must be corrected, which will be done in the next step.

\subsection{True barrier function}
Here, we correct~$\omer$ via an affine function such that it will finally give us the desired super-solution property to approximative problem. For this purpose, let $\bk\in \mathbb{R}^d$, $c\in \mathbb{R}$, $r_0>0$ and $q\in (0, r_0^{d-1})$ be arbitrary. For all $\bx \in \mathbb{R}^d \setminus B_{r_0}(\b0)$, we define
\begin{equation}\label{w2.2}
\ver(\bx)\coloneqq \omer(\bx) +\bk \cdot \bx+c.
\end{equation}
The key properties of the function~$\ver$ are formulated in the following lemma.

\begin{Lemma}\label{L4}
For every $K>0$ there exists a constant $M>0$ depending only on~$F$, $\lambda_0$  and~$K$ such that for all $\bk \in B_K(\b0)$, all $c\in \mathbb{R}$, all $r_0>0$, all $\lambda\ge \lambda_0$ and all $q\in (0, r_0^{d-1})$, the function~$\ver$ defined in~\eqref{w2.2} satisfies the inequalities
\begin{equation}
\label{super}
\begin{split}
-\diver \big(a_{\lambda}(|\nabla \ver(\bx)|)\nabla \ver (\bx) \big) &\ge 0,\\
-\Delta \ver(\bx)&\ge 0
\end{split}
\end{equation}
for all $\bx \in \mathbb{R}^d \setminus B_{r_0}(\b0)$ fulfilling $\ber(|\bx|)\ge M$ with $\ber$ given by~\eqref{DFb}.
\end{Lemma}

\begin{proof}
First, it obviously follows from \eqref{w2.2} that $\Delta \ver =\Delta \omer$ in $\mathbb{R}^d \setminus B_{r_0}(\b0)$. Hence, the second inequality in \eqref{super} is a consequence of \eqref{subhar}.

Therefore, it remains to check the first inequality in \eqref{super}. To do so, we first note that for all $\bx\in\mathbb{R}^{d}\setminus B_{r_0}(\b0)$
\begin{equation}
\nabla \ver(\bx) = \nabla \omer(\bx)+ \bk=\ber(|\bx|)\frac{\bx}{|\bx|} + \bk. \label{nabv}
\end{equation}
Consequently, a direct computation leads to
\begin{equation*}
\begin{split}
|\nabla \ver(\bx)|^2 &= (\ber)^2(|\bx|) + |\bk|^2 + 2 \ber(|\bx|) \frac{\bk \cdot \bx}{|\bx|},\\
\nabla |\nabla \ver(\bx)|&=\frac{\ber (|\bx|)(\ber)'(|\bx|)\frac{\bx}{|\bx|}+(\ber)'(|\bx|) \frac{\bx}{|\bx|}\frac{\bk \cdot \bx}{|\bx|}+\ber(|\bx|)\big( \frac{\bk}{|\bx|}-\frac{(\bk \cdot \bx) \bx}{|\bx|^3}\big) }{|\nabla \ver(\bx)|}.
\end{split}
\end{equation*}
Hence, using these identities, we obtain the following auxiliary results that will be used later
\begin{align}\label{ll1}
\nabla |\nabla \ver(\bx)| \cdot \frac{\bx}{|\bx|}&=(\ber)'(|\bx|)\frac{\ber(|\bx|)  +  \frac{\bk \cdot \bx}{|\bx|}}{|\nabla \ver(\bx)|}
\end{align}
and
\begin{equation}
\begin{split}\label{ll2}
\nabla |\nabla \ver(\bx)| \cdot \bk &=\frac{\ber(|\bx|) (\ber)'(|\bx|)\frac{\bx \cdot \bk}{|\bx|} +(\ber)'(|\bx|) \frac{(\bk \cdot \bx)^2}{|\bx|^2} + \ber(|\bx|) \big(\frac{|\bk|^2}{|\bx|}- \frac{(\bk \cdot \bx)^2}{|\bx|^3}\big)}{|\nabla \ver(\bx)|}.
\end{split}
\end{equation}

With the help of the above identities, we evaluate the left hand side of \eqref{super}. We introduce the abbreviation
\begin{align*}
L(\bx) & \coloneqq -\diver \big( a_{\lambda}(|\nabla \ver(\bx)|)\nabla \ver(\bx) \big) \\
  & = - \nabla a_{\lambda}(|\nabla \ver(\bx)|) \cdot \nabla \ver(\bx) - a_{\lambda}(|\nabla \ver(\bx)|) \diver \big( \nabla \ver(\bx) \big)\\
   &\eqqcolon L_1(\bx) + L_2(\bx).
\end{align*}
Employing~\eqref{nabv},~\eqref{ll1} and~\eqref{ll2}, we first calculate
\begin{align*}
L_1(\bx) & = - a'_{\lambda}(|\nabla \ver(\bx)|) \nabla |\nabla \ver(\bx)| \cdot \nabla \ver(\bx) \\
  & = - a'_{\lambda}(|\nabla \ver(\bx)|) \nabla |\nabla \ver(\bx)| \cdot \Big[ \ber(|\bx|) \frac{\bx}{|\bx|} + \bk \Big]\\
  & = - \frac{a'_{\lambda}(|\nabla \ver(\bx)|)}{|\nabla \ver(\bx)|} \bigg[ \ber(|\bx|)(\ber)'(|\bx|) \Big(\ber(|\bx|)  +  \frac{\bk \cdot \bx}{|\bx|} \Big) \\
  & \qquad + \ber(|\bx|) (\ber)'(|\bx|)\frac{\bx \cdot \bk}{|\bx|} +(\ber)'(|\bx|) \frac{(\bk \cdot \bx)^2}{|\bx|^2} + \ber(|\bx|) \Big(\frac{|\bk|^2}{|\bx|}- \frac{(\bk \cdot \bx)^2}{|\bx|^3}\Big) \bigg]\\
  & = \frac{a'_{\lambda}(|\nabla \ver(\bx)|)}{|\nabla \ver(\bx)|} (\ber)'(|\bx|) \Big(|\bx|-\frac{\ber(|\bx|)}{(\ber)'(|\bx|)}\Big) \Big(\frac{|\bk|^2}{|\bx|}- \frac{(\bk \cdot \bx)^2}{|\bx|^3}\Big)\\
 &\quad-a'_{\lambda}(|\nabla \ver(\bx)|)(\ber)'(|\bx|)|\nabla \ver(\bx)|.
\end{align*}
Next, taking into account once again~\eqref{nabv}, the relation~\eqref{omerp} and the fact that $\omer$ solves equation~\eqref{P1v}, we find
\begin{align*}
L_2(\bx)& = - a_{\lambda}(|\nabla \ver(\bx)|)  \diver \bigg(\frac{\a(|\nabla \omer(\bx)|)\nabla \omer(\bx)}{\a(|\nabla \omer(\bx)|)}\bigg)\\
 & = \frac{a_{\lambda}(|\nabla \ver(\bx)|)\a'(|\nabla \omer(\bx)|)}{\a(|\nabla \omer(\bx)|)} \nabla |\nabla \omer(\bx)| \cdot \nabla \omer(\bx)\\
 & = \frac{a_{\lambda}(|\nabla \ver(\bx)|)\a'(\ber(|\bx|))}{\a(\ber(|\bx|))} (\ber)'(|\bx|) \ber(|\bx|).
\end{align*}
In conclusion, after a simple algebraic manipulation, we have
\begin{equation}\label{L}
\begin{aligned}
 L(\bx) & = \frac{a'_{\lambda}(|\nabla \ver(\bx)|)}{|\nabla \ver(\bx)|} (\ber)'(|\bx|) \Big(|\bx|-\frac{\ber(|\bx|)}{(\ber)'(|\bx|)}\Big) \Big(\frac{|\bk|^2}{|\bx|}- \frac{(\bk \cdot \bx)^2}{|\bx|^3}\Big)\\
  & - a_{\lambda}(|\nabla \ver(\bx)|) (\ber)'(|\bx|) \bigg(\frac{a'_{\lambda}(|\nabla \ver(\bx)|)|\nabla \ver(\bx)|}{a_\lambda(|\nabla \ver(\bx)|)}-\frac{\a'(\ber(|\bx|))}{\a(\ber(|\bx|))}  \ber(|\bx|)\bigg).
\end{aligned}
\end{equation}

We now focus on estimating the resulting term and show its nonnegativity.  To this end, we first relate~$\ber(|\bx|)$ and~$|\nabla \ver (\bx)|$ and provide some basic estimates for sufficiently large values of~$\ber(|\bx|)$. Since $|\bk|\le K$, we deduce from~\eqref{nabv} that for $M_1 \coloneqq 2 K >0$ there holds
\begin{equation}
\ber(|\bx|)\ge M_1 \implies \ber(|\bx|)\le 2|\nabla \ver (\bx)|\le 4\ber(|\bx|).
\label{dist}
\end{equation}

>From now on, we drop the indices, since they do not vary at this point and only make the calculation look complicated. Also, we will not write the $\bx$-dependence of $v$ and $b$ explicitly. Therefore, the formula \eqref{L} reduces to
\begin{equation}\label{Ls}
\begin{aligned}
 L(\bx) & = \frac{a'_{\lambda}(|\nabla v|)}{|\nabla v|} b' \Big(|\bx|-\frac{b}{b'}\Big) \Big(\frac{|\bk|^2}{|\bx|}- \frac{(\bk \cdot \bx)^2}{|\bx|^3}\Big)\\
 &  - a_{\lambda}(|\nabla v|) b' \bigg(|\nabla v| \frac{a'_{\lambda}(|\nabla v|)}{a_{\lambda}(|\nabla v|)}-b \frac{\a'(b)}{\a(b)}  \bigg) \eqqcolon \tilde{L}_1(\bx) +  \tilde{L}_2(\bx).
\end{aligned}
\end{equation}
In the study of the sign of $L(\bx)$, we distinguish two cases - either $a'_\lambda(|\nabla v|)\le 0$, or $a'_\lambda(|\nabla v|)> 0$.

\textbf{Case $a'_{\lambda}(|\nabla v|)\le 0$}. We first focus on the term $\tilde{L}_1$. We note that $b$ is positive decreasing and therefore $-b'$ is non-negative. Also, by the use of the Cauchy-Schwarz inequality,
\begin{equation}\label{kx}
0 \le |\bk|^2 \left(1 - \frac{(\bk \cdot \bx)^2}{|\bk|^2|\bx|^2} \right) \le |\bk|^2 \le K^2
\end{equation}
and we see that both expressions in brackets in $\tilde{L}_1$ are non-negative. Therefore, $\tilde{L}_1(\bx)\ge 0$.

Concerning the sign of~$\tilde{L}_2(\bx)$,   $a_{\lambda}$ is non-negative and $b'$ is non-positive. Therefore, we need to focus on the large bracket. First we rewrite it in terms of $F$ instead of $a$ using~\eqref{af},
\begin{equation}\label{L22}
\frac{a'_{\lambda}(|\nabla v|)|\nabla v|}{a_{\lambda}(|\nabla v|)}-\frac{\a'(b)}{\a(b)}  b
= |\nabla v| \frac{F''_{\lambda}(|\nabla v|)}{F'_{\lambda}(|\nabla v|)} - b \frac{\F''(b)}{\F'(b)}.
\end{equation}
Then the value of~\eqref{L22} is non-negative for sufficiently large values of $b$ (that means, comparable with $|\nabla v|$). More precisely, we use \eqref{A2b} to get
$$
|\nabla v| \frac{F_\lambda''(|\nabla v|)}{F_\lambda'(|\nabla v|)} \ge \frac{1}{|\nabla v|^{1-\delta}}\ge \frac{1}{2^{1-\delta} b^{1-\delta}}
$$
for $|\nabla v|\ge \lambda_0$ and $b\ge M_1$. It follows from \eqref{dist} that we just require that $b\ge \max\{M_1,2\lambda_0\}$. Next,  using the definition of $\F'$ (see \eqref{ag}) we have
$$
b \frac{\F''(b)}{\F'(b)} = \frac{1}{b+1}\le \frac{1}{b}.
$$
Consequently, we see that
$$
|\nabla v| \frac{F''_{\lambda}(|\nabla v|)}{F'_{\lambda}(|\nabla v|)} - b \frac{\F''(b)}{\F'(b)}\ge \frac{1}{b} (2^{\delta-1}b^{\delta}-1).
$$
Therefore, if we set
\begin{equation}
\label{defMa}
M_2:=\max\{2\lambda_0, M_1, 2^{\frac{1-\delta}{\delta}}\},
\end{equation}
and consider that $b\ge M_2$, we obtain nonnegativity of $L$.

\textbf{Case $a'_\lambda(|\nabla v|)> 0$}. Using the discussion above, we realize that now $\tilde{L}_1(\bx)\le 0$. However, not everything is lost, as we will shortly see that the term $\tilde{L}_2(\bx)$ can in this case dominate in such way that $L(\bx)$ will finally be non-negative. To prove that, we need to dive into the study of $L(\bx)$ a bit deeper.

Let us first split the term $\tilde{L}_2(\bx)$ further into two parts $l_2(\bx)$ and $l_3(\bx)$, so that
\begin{equation*}
\tilde{L}_2(\bx) = - a'_\lambda(|\nabla v|) b' |\nabla v| + a_\lambda(|\nabla v|) b' b \frac{\a'(b)}{\a(b)} \eqqcolon l_2(\bx) + l_3(\bx).
\end{equation*}
Due to the investigation provided above ($a_\lambda$ is positive, $b'$ is non-positive and \eqref{propag} holds) we immediately see that $l_2(\bx)\ge 0$ and  $l_3(\bx)\ge 0$. We will remember the latter and use the $l_2(\bx)$ in combination with $\tilde{L}_1(\bx)$ to get
\begin{equation}\label{Er3}
L=\tilde{L}_1 + \tilde{L}_2 \ge l_2 + \tilde{L}_1= - \frac{a_\lambda'(|\nabla v|)}{|\nabla v|} b' \Bigg(|\nabla v|^2 - \Big(|\bx|-\frac{b}{b'}\Big) \Big(\frac{|\bk|^2}{|\bx|}- \frac{(\bk \cdot \bx)^2}{|\bx|^3}\Big) \Bigg)
\end{equation}
and this sum is non-negative provided so is the expression in the large bracket. With the use of the definition of $b$, see \eqref{DFb} and the identity \eqref{ag}, we see that
$$
\frac{b(r)}{1+b(r)}=\frac{q}{r^{d-1}}.
$$
Hence applying derivative with respect to $r$, we obtain
$$
\frac{b'}{(1+b)^2} = -(d-1) \frac{q}{r^d} = -\frac{d-1}{r} \frac{q}{r^{d-1}}=-\frac{d-1}{r}\frac{b}{1+b}.
$$
Thus, after a simple algebraic manipulation, we get
\begin{equation}\label{Er1}
-\frac{b}{b'}=\frac{r}{(d-1)(1+b)}\le r.
\end{equation}
Finally, this estimate in \eqref{Er3}, we see that
$$
|\nabla v|^2 - \Big(|\bx|-\frac{b}{b'}\Big) \Big(\frac{|\bk|^2}{|\bx|}- \frac{(\bk \cdot \bx)^2}{|\bx|^3}\Big) \ge |\nabla v|^2 - 2|\bk|^2\ge |\nabla \bv|^2 - 2K^2.
$$
Thus, if $b \ge \max\{M_1, 4K\}$ then it follows from \eqref{dist} that
$$
|\nabla v|^2 - \Big(|\bx|-\frac{b}{b'}\Big) \Big(\frac{|\bk|^2}{|\bx|}- \frac{(\bk \cdot \bx)^2}{|\bx|^3}\Big) \ge \frac{b^2}{4} - 2K^2\ge 2K^2\ge 0
$$
and therefore going back to \eqref{Er3}, we see that $L\ge 0$. Hence, if we set
\begin{equation}
M:=\max\{M_1, M_2, 2K\},\label{defM}
\end{equation}
then the first inequality in \eqref{super} holds true  for all $\bx \in \mathbb{R}^d \setminus B_{r_0}(\b0)$ with $\ber(|\bx|)\ge M$, and the proof of the lemma is complete.
\end{proof}

\subsection{Completion of the proof}

This part of the proof is very similar to~\cite[Section~4]{BBM}  but for readers convenience, we  describe it also here, since in our setting the computations can be done easier.

Once the true barrier function from Lemma~\ref{L4} is at our disposal, we can return to study the normal derivative, with the aim to prove an estimate of the form~\eqref{basis2}. Therefore, following the discussion in Section~\ref{barrier}, we see that for every $\bx_0\in \partial \Omega$, we need to find barriers fulfilling \eqref{ubsub} and \eqref{ubsup} for some $r>0$ with the uniform control
\begin{equation}
\|u^b\|_{1,\infty}\le C(\|u_0\|_{1,1}, \Omega, F) \label{zdeek}
\end{equation}
with the constant independent of $\bx_0$, $r$ and $\lambda$ and $\mu$. We shall finally see why the uniform exterior ball condition plays the key role for the analysis. First, since $\Omega$ is by assumption of class~$\mathcal{C}^{1}$ and satisfies an exterior ball condition, we find positive (from now \emph{fixed}) constants $r_0$, $L$, $L_d$ and~$N$ depending only on~$\Omega$ such that we can suppose that an arbitrary boundary point $\bx_0\in \partial \Omega$  is given, after an orthogonal transformation, by $\bx_0=(\b0, -r_0)$ (we use the notation $\bx=(\bx',x_d)$) and that we have the inclusions
\begin{equation*}
\begin{aligned}
\Gamma & \coloneqq \{\bx\in \mathbb{R}^d \colon |\bx'|< L, \; f(\bx')=x_d\}  \subset \partial \Omega,\\
\Omega_+ & \coloneqq \{\bx\in \mathbb{R}^d \colon |\bx'|< L, \; f(\bx') - L_d < x_d < f(\bx')\} \subset \Omega,\\
\Omega_{-} & \coloneqq \{\bx\in \mathbb{R}^d \colon |\bx'|< L, \; f(\bx')<x_d<f(\bx')+L_d\}  \subset \mathbb{R}^d \setminus\Omega,
\end{aligned}
\end{equation*}
with a function $f\in \mathcal{C}^{1}(-L,L)^{d-1}$ fulfilling $\|f\|_{1,\infty}\le N$, $f(\b0')=-r_0$ and $D_i f(\b0)=0$ for all $i=1,\ldots, d-1$. Furthermore, we can assume that~$r_0$ is fixed such that  $B_{r_0}(\b0)\subset \Omega_{-}$ holds and that for all $\bx \in \Gamma$ we have
\begin{equation}
\label{bound-dist-cond}
M^*(|\bx|-r_0)\ge |\bx-\bx_0|^2
\end{equation}
with some constant~$M^*$ depending only on~$\Omega$ and~$r_0$. From now on, we consider some fixed $\bx_0\in \partial \Omega$.

Next, we introduce the barrier function. For arbitrary $\delta \in (0,1)$ (to be specified later), we fix
$$
q:=(1-\delta)^{d-1}r_0^{d-1}
$$
and consider  functions~$\ber$ and~$\omer$ introduced in~\eqref{DFb} and~\eqref{dfomer}. Using them, we introduce  the function~$\ver$ from~\eqref{w2.2} with the  specific choices $\bk \coloneqq \nabla u_0(\bx_0)$ and $c=u_0(\bx_0) - \nabla u_0(\bx_0) \cdot \bx_0$, that is, with
\begin{equation}\label{dfvef}
\ver(\bx)=\omer(\bx) +\nabla u_0 (\bx_0) \cdot (\bx-\bx_0) + u_0(\bx_0).
\end{equation}
Note that $\ver$ is well defined outside the ball $B_{r_0}(\b0)$ and so it is well-defined also in $\Omega_+$. In addition, it is clear that $|\bk|\le \|\nabla u_0\|_{\infty}$ holds, hence, we fix $K \coloneqq \|\nabla u_0\|_{\infty}$ and also the number~$M$ (depending only on~$F$ and this~$K$) according to Lemma~\ref{L4}.

Next, we specify the maximal value of $\delta$. If we define $\delta_{\max}\in (0,1/2)$ by the relation
\begin{equation}\label{delmax}
(1-2\delta_{\max})^{d-1}:=\max\left\{\frac{M}{M+1}, \frac{ M^* \|u_0\|_{1,\infty}}{1+ M^* \|u_0\|_{1,\infty}} \right\},
\end{equation}
where $M^*$ comes from \eqref{bound-dist-cond} and $M$ from Lemma~\ref{L4}, and $r_{\max}$ as
\begin{equation}\label{rmax}
r_{\max}:=\frac{(1-\delta_{\max})r_0}{1-2\delta_{\max}}
\end{equation}
then for all $\bx \in B_{r_{\max}}(\bx_0)\setminus B_{r_0}(\bx_0)$ and arbitrary $\delta\in (0,\delta_{\max})$, we get by using \eqref{DFb} that
\begin{equation}\label{DFb2}
\begin{split}
\ber(|\bx|)&\ge \ber(r_{\max})= \frac{q}{r^{d-1}_{\max}-q}=\frac{(1-\delta)^{d-1}r_0^{d-1}}{\frac{(1-\delta_{\max})^{d-1}r^{d-1}_0}{(1-2\delta_{\max})^{d-1}}-(1-\delta)^{d-1}r_0^{d-1}}\\
&=\frac{(1-2\delta_{\max})^{d-1}}{\frac{(1-\delta_{\max})^{d-1}}{(1-\delta)^{d-1}}-(1-2\delta_{\max})^{d-1}}\ge \frac{(1-2\delta_{\max})^{d-1}}{1-(1-2\delta_{\max})^{d-1}}\\
&=\max\{M, M^* \|u_0\|_{1,\infty} \},
\end{split}
\end{equation}
where for the last inequality we used \eqref{delmax}. Consequently, using Lemma~\ref{L4}, we see that $\ver$ satisfies the first inequality in \eqref{ubsub}. Thus it remains to show that it also satisfies the second inequality and to choose $\delta$ uniformly, i.e., depending only on $u_0$, $F$ and $\Omega$ in order to get the uniform control on $|\nabla \ver|$.

Next, we want to identify a part of $\Gamma$ on which $\ver(\bx) \geq u(\bx) = u_0(\bx)$ holds, that is, where
\begin{equation}\label{dui}
\omer(\bx) +\nabla u_0(\bx_0)\cdot (\bx -\bx_0) +u_0(\bx_0)-u_0(\bx) \geq 0.
\end{equation}
>From the Taylor expansion of~$u_0$ and the~$\mathcal{C}^{1,1}$-regularity assumption on~$u_0$ we know that
$$
\begin{aligned}
&\left|u_0(\bx)-u_0(\bx_0)-\nabla u_0(\bx_0)\cdot (\bx - \bx_0)\right|\le \|u_0\|_{1,\infty} |\bx-\bx_0|^2,
\end{aligned}
$$
so to verify~\eqref{dui} it is enough to check where
\begin{equation}
\omer(\bx) - \|u_0\|_{1,\infty} |\bx - \bx_0|^2 \geq 0 \label{dui3}
\end{equation}
holds. Using the definitions of $\ber$ in~\eqref{DFb} and of~$\omer$ in~\eqref{dfomer}, combined with the fact that~$(F')^{-1}$ is monotonically increasing, we have for all $\bx\in \Gamma$
\begin{equation*}
\omer(\bx)\ge (|\bx|-r_0) (F')^{-1} \bigg(\frac{(1-\delta)^{d-1}r^{d-1}_0}{|\bx|^{d-1}}\bigg)=(|\bx|-r_0)\frac{(1-\delta)^{d-1}r^{d-1}_0}{|\bx|^{d-1}-(1-\delta)^{d-1}r^{d-1}_0}.
\end{equation*}
Consequently, in order to guarantee~\eqref{dui3} and thus~\eqref{dui} it is sufficient, in view of~\eqref{bound-dist-cond}, to have
\begin{equation*}
\frac{(1-\delta)^{d-1}r^{d-1}_0}{|\bx|^{d-1}-(1-\delta)^{d-1}r^{d-1}_0}\ge M^* \|u_0\|_{1,\infty},
\end{equation*}
which is indeed true for all~$\bx$ with $r_0 \leq |\bx| \leq r_{\max}$, by the choices of the parameter~$\delta_{\max}$ in~\eqref{delmax} and of the radius~$r_{\max}$ in~\eqref{DFb2}. Thus, we have verified
\begin{equation}
\label{boundary_values_Gamma}
u(\bx)\le\ver(\bx) \qquad \textrm{for all } \bx \in \Gamma \textrm{ with } r_0 \leq |\bx|\leq r_{\max}.
\end{equation}

We have already verified that $\ver$ is a supersolution, i.e., satisfies \eqref{ubsub}$_1$ in $B_{r_{\max}}(\bx_0)$ and that also fulfills \eqref{ubsub}$_2$ on $\Gamma\cap B_{r_{\max}}(\bx_0)$. Hence, to show the validity of \eqref{ubsub}, we need to show that it also satisfies \eqref{ubsub}$_2$ on $\partial (B_{r_{\max}}(\bx_0) \cap \Omega)$.  This shall now be done  by a proper choice of a local neighborhood and of~$\delta \in (0,\delta_{\max})$.

Since~$r_0$ and $r_{\max}$ are already fixed (and depend only on on~$\Omega$, $F$  and~$u_0$), we can find two constants~$L^*$ and $L^*_{d} \leq L_d$ sufficiently small such that
\begin{equation*}
\begin{aligned}
\Gamma^*& \coloneqq \{\bx\in \mathbb{R}^d \colon |\bx'|< L^*, \; f(\bx')=x_d\} \subset \partial \Omega,\\
\Omega^*_+& \coloneqq \{\bx\in \mathbb{R}^d \colon |\bx'|< L^*, \; f(\bx') - L^*_d < x_d < f(\bx'\}  \subset \Omega \cap (B_{r_{\max}}\setminus B_{r_0}),\\
\Omega^*_{-}& \coloneqq \{\bx\in \mathbb{R}^d \colon |\bx'|< L^*,  \; f(\bx')<x_d<f(\bx')+L_d\} \subset \mathbb{R}^d \setminus\Omega.
\end{aligned}
\end{equation*}
Having introduced this notation, we see with help of~\eqref{DFb2} and~\eqref{boundary_values_Gamma} that $\ver$ solves \eqref{ubsub}$_1$ in the relative neighborhood~$\Omega^*_+$ of~$\bx_0$ and satisfies $\ver \geq u$ on~$\Gamma^*$, for all $\delta < \delta_{\max}$. Our goal is to show $\ver\ge u$ on the remaining part of $\partial \Omega_+^*$.

We start the proof by noting that there exists a positive $\eta>0$ (independent of $\delta$) such that
$$
|\bx|\ge r_0 + \eta
$$
for all $\bx \in \partial \Omega_+^* \setminus \Gamma^*$, which follows from the choices of~$r_0$ and~$\Omega_+^*$.
Hence, from the definition of~$\ver$ in~\eqref{dfvef} and of~$\omer$ in~\eqref{dfomer} we find
\begin{equation}
\ver(\bx) \ge \int_{r_0}^{r_0 +\eta} \ber (r) \dr - C^*(\Omega)  \|u_0\|_{1,\infty}. \label{nemohu}
\end{equation}
On the other hand, we know $\|u\|_{\infty} \le \|u_0\|_{\infty}$ from~\eqref{AE3}. Thus, in order to show that $\ver \ge u$ on $\partial \Omega_+^*$, we need to choose  $\delta\in (0,\delta_{\max})$ such that
\begin{equation}
\int_{r_0}^{r_0 +\eta} \ber (r) \dr \ge C^*(\Omega)  \|u_0\|_{1,\infty} + \|u_0\|_{\infty}.\label{konec}
\end{equation}
Using the definition of~$\ber$ in~\eqref{DFb} and the substitution formula, we deduce that
\begin{align*}
\int_{r_0}^{r_0 +\eta} \ber (r) \dr &=\int_{r_0}^{r_0 + \eta}\frac{(1-\delta)^{d-1}r^{d-1}_0}{r^{d-1}-(1-\delta)^{d-1}r^{d-1}_0}\dr\\
&=r_0\int_{1}^{1+\frac{\eta}{r_0}}\frac{(1-\delta)^{d-1}}{t^{d-1}-(1-\delta)^{d-1}}\dt\\
&\ge \frac{r_0(1-\delta)^{d-1}}{\left(1+\frac{\eta}{r_0}\right)^{d-2}}\int_{1}^{1+\frac{\eta}{r_0}}\frac{t^{d-2}}{t^{d-1}-(1-\delta)^{d-1}}\dt\\
&\ge\frac{r_0(1-\delta_{\max})^{d-1}}{(d-1)\left(1+\frac{\eta}{r_0}\right)^{d-2}}\ln\left(\frac{(1+\frac{\eta}{r_0})^{d-1}-(1-\delta)^{d-1}}{1-(1-\delta)^{d-1}}\right).
\end{align*}
It is important to notice that the right hand side tends to $\infty$ as $\delta \to 0_+$. Consequently, using the above inequality, we can fix $\delta\in (0, \delta_{\max})$ (depending only on~$\Omega$,~$F$ and $u_0$) such that \eqref{konec} holds true. Hence, we have constructed the upper barrier  function $u^b$ on the relative neighborhood~$\Omega_+$ of~$\bx_0$.  The lower barrier $u_b$ is, however, constructed in the very similar manner. As the domain is of class $C^1$ and satisfies the uniform exterior ball condition, by the same procedure we can prove existence of both upper and lower barriers in every point on the boundary $\partial \Omega$. Thus, these can be merged to form an upper and lower barrier for the function $u_{\lambda,\mu}$ on whole $\partial \Omega$.

To summarize, thanks to the existence of barriers, we obtain \eqref{wened}. Since the $W^{1,\infty}$-norm of barriers depends only on $F$, $u_0$ and $\Omega$ and most importantly is independent of $\lambda$ and $\mu$. Hence, using (\ref{supersolu}), we can choose  $\lambda$ sufficiently large in order to have \eqref{muin}, which implies \eqref{dream} and finishes the proof.

%\bibliographystyle{amsplain}
%\bibliography{literature}

\begin{thebibliography}{10}

\bibitem{BBM}
L.~Beck, M.~Bul\'{\i}\v{c}ek, and E.~Maringov\'{a}, \emph{Globally {L}ipschitz
  minimizers for variational problems with linear growth}, accepted to ESAIM:
  COCV, doi: 10.1051/cocv/2017065, 2018.

\bibitem{Cellina}
A.~Cellina, \emph{On the bounded slope condition and the validity of the
  {E}uler {L}agrange equation}, SIAM J. Control Optim. \textbf{40} (2001/02),
  no.~4, 1270--1279.

\bibitem{Cellina2}
\bysame, \emph{Comparison results and estimates on the gradient without strict
  convexity}, SIAM J. Control Optim. \textbf{46} (2007), no.~2, 738--749.

\bibitem{Clarke}
F.~Clarke, \emph{Continuity of solutions to a basic problem in the calculus of
  variations}, Ann. Sc. Norm. Super. Pisa Cl. Sci. (5) \textbf{4} (2005),
  no.~3, 511--530.

\bibitem{DMP}
A.~Dall'Aglio, E.~Mascolo, and G.~Papi, \emph{Local boundedness for minima of
  functionals with nonstandard growth conditions}, Rend. Mat. Appl. (7)
  \textbf{18} (1998), no.~2, 305--326.

\bibitem{GT}
D.~Gilbarg and N.~S. Trudinger, \emph{Elliptic partial differential equations
  of second order}, Classics in Mathematics, Springer-Verlag, Berlin, 2001,
  Reprint of the 1998 edition.

\bibitem{GLM}
S.~Granlund, P.~Lindqvist, and O.~Martio, \emph{Note on the {PWB}-method in the
  nonlinear case}, Pacific J. Math. \textbf{125} (1986), no.~2, 381--395.

\bibitem{HKM}
J.~Heinonen, T.~Kilpel\"ainen, and O.~Martio, \emph{Nonlinear potential theory
  of degenerate elliptic equations}, Dover Publications, Inc., Mineola, NY,
  2006, Unabridged republication of the 1993 original.

\bibitem{L1}
H.~Lebesgue, \emph{Sur des cas dimpossibilit\'e du probl\'{e}me de {D}irichlet
  ordinaire {V}ie de la soci\'et\'e}, Bull. Soc. Math. Fr. \textbf{41} (1913),
  no.~17, 1--62.

\bibitem{L2}
\bysame, \emph{Conditions de r\'egularit\'e, conditions d' irr\'egularit\'e,
  conditions d'impossibilit\'e dans le probl\'eme de {D}irichlet}, C. R. Acad.
  Sci. Paris \textbf{178} (1924), 349--354.

\bibitem{Ma}
P.~Marcellini, \emph{Everywhere regularity for a class of elliptic systems
  without growth conditions}, Ann. Scuola Norm. Sup. Pisa Cl. Sci. (4)
  \textbf{23} (1996), no.~1, 1--25.

\bibitem{MT}
C.~Mariconda and G.~Treu, \emph{Local {L}ipschitz regularity of minima for a
  scalar problem of the calculus of variations}, Commun. Contemp. Math.
  \textbf{10} (2008), no.~6, 1129--1149.

\bibitem{Perron}
O.~Perron, \emph{Eine neue {B}ehandlung der ersten {R}andwertaufgabe f\"ur
  {$\Delta u=0$}}, Math. Z. \textbf{18} (1923), no.~1, 42--54.

\bibitem{S}
G.~Stampacchia, \emph{On some regular multiple integral problems in the
  calculus of variations}, Comm. Pure Appl. Math. \textbf{16} (1963), 383--421.

\end{thebibliography}
%
%\end{document}

\providecommand{\bysame}{\leavevmode\hbox to3em{\hrulefill}\thinspace}
\providecommand{\MR}{\relax\ifhmode\unskip\space\fi MR }
% \MRhref is called by the amsart/book/proc definition of \MR.
\providecommand{\MRhref}[2]{%
  \href{http://www.ams.org/mathscinet-getitem?mr=#1}{#2}
}
\providecommand{\href}[2]{#2}

\end{document}